\documentclass[11pt,twoside]{amsart}
\usepackage{amssymb}
\usepackage{amsmath}
\usepackage{graphicx,color}

\newtheorem{thm}{Theorem}[section]
\newtheorem{prop}[thm]{Proposition}

\newtheorem{lem}[thm]{Lemma}
\newtheorem{cor}[thm]{Corollary}
\newtheorem{defn}[thm]{Definition}

\newtheorem{ex}[thm]{Example}
\newtheorem{rem}[thm]{Remark}
\newtheorem{conj}[thm]{Conjecture}
\newtheorem{prob}[thm]{Problem}

\newcommand{\skipit}[1]{{}}
\newcommand{\prfend}{\hbox to7pt{\hfil}
\par\vskip-\baselineskip\hbox to\hsize
{\hfil\vbox {\hrule width6pt height6pt}}\vskip\baselineskip}

\newcommand{\ZZ}{\mathbb{Z}}

\newcommand {\PP}{\mathbb{P}}

\newcommand{\cL}{\mathcal{L}}
\newcommand{\cP}{\mathcal{P}}
\newcommand{\cA}{\mathcal{A}}
\newcommand{\cB}{\mathcal{B}}

\newcommand{\cO}{\mathcal{O}}

\newcommand{\cH}{\mathcal{H}}
\newcommand{\cD}{\mathcal{D}}
\newcommand{\cF}{\mathcal{F}}

\DeclareMathOperator{\Image}{Im}

\DeclareMathOperator{\Proj}{Proj}

\DeclareMathOperator{\degree}{degree}

\DeclareMathOperator{\depth}{depth}
\DeclareMathOperator{\pd}{pd}
\DeclareMathOperator{\rk}{rank}

\DeclareMathOperator{\codim}{codim}

\newcommand{\myarrow}[2]{\hbox to #1pt{\hfil$\to$\hfil}{\hskip-#1pt{\raise
10pt\hbox to#1pt{\hfil$\scriptscriptstyle #2$\hfil}}}}

\begin{document}

\title{Togliatti systems and Galois coverings}
\author[Emilia Mezzetti]{Emilia Mezzetti}
\address{Dipartimento di Matematica e  Geoscienze, Sezione di Matematica e Informatica, Universit\`a di
Trieste, Via Valerio 12/1, 34127 Trieste, Italy}
\email{mezzette@units.it}

\author[Rosa M. Mir\'o-Roig]{Rosa M. Mir\'o-Roig}
\address{Facultat de Matem\`atiques i Inform\`{a}tica,
Departament de Matem\`{a}tiques i Inform\`{a}tica, Gran Via de les Corts Catalanes
585, 08007 Barcelona, Spain}
\email{miro@ub.edu}

\begin{abstract} We study the homogeneous artinian ideals of the polynomial ring $K[x,y,z]$ generated by the   homogenous polynomials of degree $d$ which are invariant under an action of the cyclic group $\ZZ/d\ZZ$, for any $d\geq 3$. We prove that they are all monomial Togliatti systems, and that they are minimal if the action is defined by a diagonal matrix having on the diagonal $(1, e, e^a)$, where $e$ is a primitive $d$-th root of the unity. We get a complete description when  $d$ is prime or a power of a prime. We also establish the relation of these systems with linear Ceva configurations.
\end{abstract}

\thanks{Acknowledgments:   The first author is member of INdAM - GNSAGA and is supported by PRIN
``Geometria delle variet\`a algebriche'' and by
FRA, Fondi di Ricerca di Ateneo, Universit\`a di Trieste. The second  author was partially   supported
by  MTM2013-45075-P.
}

\maketitle

\tableofcontents

\markboth{E. Mezzetti, R. M. Mir\'o-Roig}{Togliatti
systems and Galois coverings}

\today

\large

\section{Introduction}

The Togliatti surface in $\PP^5$ is the rational surface parametrized by the cubic monomials in three variables $x^2y, x^2 z, xy^2, xz^2, y^2 z, yz^2$. It was introduced and studied by Eugenio Togliatti in his two articles \cite{T1}, \cite{T2} about rational surfaces satisfying Laplace equations. The apolar  system of cubics, i.e. the ideal $I_3$ generated by the cubic monomials $x^3, y^3, z^3, xyz$,  has the remarkable property of being the only homogeneous ideal, generated by four cubics of the form $x^3, y^3, z^3, f$, failing the Weak Lefschetz Property (\cite{BK}). This twofold example has lead the authors of this note, together with G. Ottaviani, to establish the connection between the geometric notion of variety satisfying Laplace equations and the algebraic notion of artinian ideal failing the Weak Lefschetz Property, thanks to the notion of Togliatti system,  introduced  in \cite{MMO}, see Definition \ref{togliattisystem}.

Togliatti systems of cubics have been completely described in \cite{MMO} and \cite{MRM}. For degree $d>3$, the picture quickly becomes much more complicated and only partial results are known so far, see for instance \cite{MM} containing a characterization of Togliatti systems with \lq\lq low'' number of generators.

An interesting property of the ideal $I_3$ is that the associated morphism $\varphi: \PP^2 \rightarrow \PP^3$ is a cyclic Galois cover of degree $3$ of the image surface, such that for a general line $L\subset \PP^2$ the inverse image of $\varphi(L)$ is a union of three lines in general position. This observation has been exploited to give a new beautiful  proof of the theorem of Togliatti
(see \cite{V2}, Theorem 2.2.1).

In this article we start from this observation to construct a new class of examples of Togliatti systems in three variables of any degree $d$, and by consequence of rational surfaces parametrized by polynomials of degree $d$ satisfying a Laplace equation of order $d-1$. We call them GT-systems in honour of Galois and Togliatti. Precisely, we consider the Galois cyclic covers with domain $\PP^2$, with cyclic group $\ZZ/d\ZZ$, for any $d\geq 3$ and any representation of $\ZZ/d\ZZ$ on $GL(3,K)$. We prove in Theorem \ref{upper} that they are all defined by monomial Togliatti systems of degree $d$. In other words, the ideal generated by the invariant homogeneous  polynomials of degree $d$ for any action of this type is a monomial Togliatti system. Moreover we prove in Theorem \ref{minimal} that the GT-systems are all minimal in the case of an action represented by a matrix of the form $M_a=\begin{pmatrix}
1&0&0\\
0&e&0\\
0&0&e^a
\end{pmatrix}$, where $e$ denotes a primitive $d$-th root of the unity and $GCD(a,d)=1$. In particular this is always the case when $d$ is prime or a power of a prime.

We then perform a detailed study of the Togliatti systems associated to the representation $M_a$, for any coprime $a$ and $d$, giving a complete description of the  classes and number of generators of the ideals of invariant degree $d$ polynomials $I_a$. We also describe for any degree $d$ the geometry of the toric surface $S_d$, image of the morphism defined by the monomials invariant for the action of $M_2$. These we call generalized classical Togliatti systems.

Finally, we also explain the interesting relations  between GT-systems and linear configurations in $\PP^2$, in particular Ceva configurations and their dual: Fermat arrangements.

\vskip 2mm
Let us briefly explain how this paper is organized. In Section \ref{prelim}, we fix the notation and basic facts needed later on. In particular, we recall the definition of smooth monomial minimal Togliatti system and we introduce new families of Togliatti systems, the so-called GT-systems, which will be our objects of study.
In Section \ref{Galois}, we establish the basic properties of a GT-system $I$ generated by all forms of degree $d$ invariants under the action of the matrix $M_{a,b,c}
=\begin{pmatrix}
e^a&0&0\\
0&e^b&0\\
0&0&e^c
\end{pmatrix}$, where $e$ denotes a primitive $d$-th root of the unity and $GCD(a,b,c,d)=1$.
Namely, $I$ is generated by monomials, its minimal number of generators $\mu (I)$ is bounded by $d+1$, it fails Weak Lefschetz Property from degree $d-1$ to degree $d$ and  it defines a Galois cover $\varphi _I:\PP^2 \longrightarrow \PP^{\mu(I)-1}$ with cyclic Galois group $\ZZ/d\ZZ$. Section \ref{minimality} is entirely devoted to study the minimality of GT-systems. As a main tool we use circulant matrices. In Section \ref{d prime}, we completely classify GT-systems in the case $d$ prime or a power of a prime, while in Section \ref{general} we give the complete classification for all $d$ of the actions represented by matrices of the form $M_a.$
In Section \ref{geomproperties}, we study from a geometric point of view the rational surfaces $S_d$ associated to  generalized classical GT-systems. We prove that the homogeneous ideal $I(S_d)$ is a Cohen-Macaulay ideal generated by quadrics and cubics if $d$ is odd and only by quadrics if $d$ is even. Finally, we describe the singular locus of $S_d$: the 3 fundamental points of $ \PP^2$ are sent to the singular points of $S_d$ which are cyclic quotient singularities.
Finally,  in Section \ref{config}, we establish the link between GT-systems and Ceva linear configurations $C(d)$; and we study the freeness of the arrangement $\cH _d$ of lines associated to $C(d)$.

\vskip 2mm
\noindent {\em Acknowledgement.} Many of the ideas for this paper were developed during a stay at BIRS (Banff International Research Station) and
the authors are very
grateful to BIRS for the financial support. We also thank J. Vall\`{e}s, F. Perroni, R. Pardini, A. Logar, M. Reid and A. Iarrobino
for interesting comments and conversations.


\section{Preliminaries}\label{prelim}

In this section we establish general results on Weak Lefschetz Property and linear configurations that will be used
along the remainder of this paper. Throughout this work  $K$ will be an
algebraically closed field of
characteristic zero
and $R=K[x_0,x_1,\ldots ,x_n]$.

\subsection{Togliatti systems}
\begin{defn}\label{def of wlp}\rm Let $I\subset R$ be a
homogeneous artinian ideal. We  say that   $R/I$
 has the {\em Weak Lefschetz Property} (WLP, for short)
if there is a linear form $L \in (R/I)_1$ such that, for all
integers $j$, the multiplication map
\[
\times L: (R/I)_{j} \to (R/I)_{j+1}
\]
has maximal rank, i.e.\ it is injective or surjective.
We will often abuse notation and say that the ideal $I$ has the
WLP.
\end{defn}

 Associated to any artinian
ideal $I$
generated by $r$
forms
  $F_1, \ldots ,F_r\in R$ of  degree $d$  there is a
morphism
$$\varphi _{I}:\PP^n \longrightarrow \PP^{r-1}.$$
Its image $X_{n,I_d}:={\Image (\varphi _{I_d})}\subset \PP^{r-1}$
is the projection of the $n$-dimensional Veronese variety $V(n,d)$
from the linear system $\langle(I^{-1})_d \rangle\subset \mid \cO
_{\PP^n}(d)\mid=R_d$ where
 $I^{-1}$ is the ideal generated by the Macaulay inverse system of $I$
(See \cite{MMO}, \S 3 for details).  Analogously, associated to
  $(I^{-1})_d$ there is a rational map
$$\varphi _{(I^{-1})_d}:\PP^n \dashrightarrow \PP^{{n+d
\choose d}-r-1}.$$ The closure of its image
$X_{n,(I^{-1})_d}:=\overline{\Image (
\varphi _{(I^{-1})_d})}\subset \PP^{{n+d\choose d}-r-1}$
is the projection of the $n$-dimensional Veronese variety
$V(n,d)$ from the linear system $\langle F_1,\ldots ,F_r \rangle \subset
\mid \cO _{\PP^n}(d)\mid=R_d$.

 We have (\cite{MMO}):

\begin{thm}\label{teathm} Let $I\subset R$ be an artinian
ideal
generated
by $r$ homogeneous polynomials $F_1,...,F_{r}$ of degree $d$.
If
$r\le {n+d-1\choose n-1}$, then
  the following conditions are equivalent:
\begin{itemize}
\item[(1)] The ideal $I$ fails the WLP in degree $d-1$,
\item[(2)] The  homogeneous forms $F_1,\ldots, F_{r}$ become
$k$-linearly dependent on a general hyperplane $H$ of $\PP^n$,
\item[(3)] The $n$-dimensional   variety
 $X=X_{n,(I^{-1})_d}$ satisfies at least one Laplace equation of order
$d-1$.
\end{itemize}
\end{thm}

\begin{proof} \cite{MMO}, Theorem 3.2.
\end{proof}

The above result motivates the following definition:

\begin{defn} \label{togliattisystem} \rm Let $I\subset R$ be an artinian
ideal
generated
by $r$ forms $F_1, \ldots ,F_{r}$ of degree $d$, $r\le {n+d-1\choose n-1}$.
We introduce the following definitions:
\begin{itemize}
\item[(i)] $I$  is a \emph{Togliatti system}
if it satisfies the three equivalent conditions in  Theorem \ref{teathm}.

\item[(ii)]   $I$ is a \emph{monomial Togliatti system} if,
in addition, $I$ (and hence $I^{-1}$) can be generated by monomials.

\item[(iii)]   $I$ is a \emph{smooth Togliatti system} if,
in addition, the $n$-dimensional   variety
$X$ is smooth.

\item[(iv)] A  monomial  Togliatti system $I$ is said to be
\emph{minimal} if $I$ is generated by monomials $m_1, \ldots ,m_r$ and
there is no proper subset $m_{i_1}, \ldots ,m_{i_{r-1}}$ defining a
monomial  Togliatti system.
    \end{itemize}
\end{defn}

The names are in honour of Eugenio Togliatti who proved that for $n=2$
the only
smooth  Togliatti system of cubics is
$I=(x^3,y^3,z^3,xyz)\subset K[x,y,z]$ (see \cite{BK}, \cite{T1},
\cite{T2}).

We underline an interesting geometric property of this example: the morphism
$$\varphi _{I}:\PP^2 \longrightarrow \PP^{3}$$
is a cyclic Galois covering of degree $3$ of the image, the cubic  surface  of equation
$x_0x_1x_2-x_3^3=0$. Moreover, for a general line $L\subset \PP^2$, the inverse image of
$\varphi _{I}(L)$ is a union of three lines.
This observation can be exploited to give a new beautiful  proof of the theorem of Togliatti
(see \cite{V2}, Theorem 2.2.1).

This example admits a family of generalizations in all degrees, sharing the property
of defining a cyclic Galois
covering, as we see in the next definition.

\begin{defn} \label{genclassicalTogliatti} \rm Fix an integer
$d=2k+\epsilon$, $0\le \epsilon \le 1$. The  monomial artinian ideal
$I=(x^d, y^d, z^d,xy^{d-2}z , x^2y^{d-4}z^2,\ldots ,
x^ky^{\epsilon}z^k)\subset K[x,y,z]$ defines a monomial Togliatti
system that
we call  {\em generalized classical Togliatti system}. Clearly for $d=3$ we recover
Togliatti's example.
\end{defn}

Any such ideal
defines a Galois covering of degree $d$
 $$\varphi _I:\PP^2\longrightarrow \PP^{k+2}$$
of the surface $\varphi_I(\PP^2)$, with cyclic
Galois group $\ZZ/d\ZZ$ represented by the matrix  $\begin{pmatrix}
1&0&0\\
0&e&0\\
0&0&e^2
\end{pmatrix} ,$  where $e$ is a primitive $d$-th root of $1$. This is not
true for {\em all} monomial Togliatti systems, as shown for instance by the ideal
 $I=(x^5,y^5,z^5,x^4y,x^4z)\subset K[x,y,z]$. It  is a monomial Togliatti system
and its  associated regular map $\varphi _I:\PP^2\longrightarrow\PP^4$
is  a birational morphism.

\begin{rem}\rm Note that, for $d$ odd, all monomials different from
$x^d$, $y^d$, $z^d$ in a generalized classical Togliatti system contain
all variables with strictly positive exponent. Therefore, applying the
smoothness criterion for toric varieties (see \cite{MM}, Proposition
3.4), we get that, for $d$ odd, the generalized classical Togliatti systems
are smooth. This is no longer true for $d$ even.
\end{rem}

So, we are let to
pose the following problem:

\begin{prob} \rm  Classify all
Togliatti systems
$I\subset K[x,y,z]$ generated by forms of degree $d$,
whose associated regular map is a Galois covering
with cyclic Galois group $\ZZ/d\ZZ$. We will call them {\em GT-systems.}
\end{prob}
Since any representation of $\ZZ/d\ZZ$ on $GL(3,\mathbb C)$ can be diagonalized (see for instance \cite{Se}),
we can assume that it is represented by a matrix of the form $M_{a,b,c}:=\begin{pmatrix}
e^a&0&0\\
0&e^b&0\\
0&0&e^c
\end{pmatrix} ,$  where $e$ is a primitive $d$-th root of $1$, and $GCD(a,b,c,d)=1$.

The goal of next section will be to prove our first main result that for any integer $d$, any ideal generated by all forms of degree $d$ invariant under the action of $M_{a,b,c}$ is a monomial GT-system (see Theorem 3.1, Proposition 3.2 and Theorem 3.4). In Section \ref{minimality}, we will study the minimality of GT-systems (see Theorem  \ref{minimal}).
In Section  \ref{d prime} we classify all GT-systems $I\subset
K[x,y,z]$ generated by forms of degree $d$ with $d$ a prime integer, or
a power of a prime and we postpone until \S 6 the study of GT-systems
$I\subset K[x,y,z]$ generated by forms of arbitrary degree $d$. Finally, we devote Sections \ref{config}  and \ref{geomproperties} to study their geometric properties.

\subsection{Linear configurations}

We will finish this section recalling the definition of linear configuration and we will prove later that associated to any GT-systems there always exists a linear configuration (see Theorem \ref{configurations}). The reader can see \cite{D} for more details about linear configurations in Algebraic Geometry.

\begin{defn} \rm
  A {\em  linear configuration} $C$ in $\PP^2$ is  a finite set of
points, and a finite arrangement of lines, such that each point is
incident to the same number of lines and each line is incident to the
same number of points.
A linear configuration $C$ in the plane will be denoted by $(p_s,\ell _r)$
where $p$ is the number of points, $\ell $ is the number of lines, $s$
is the number of lines per point, and $r$ is the number of points per
line. These numbers necessarily satisfy the equation $p\cdot s=\ell
\cdot r$.
\end{defn}

\begin{figure}[!h]
  \centering
  \includegraphics[width=2in,angle=90]{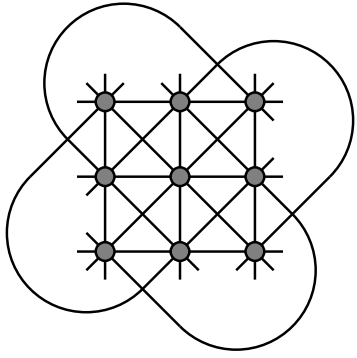}
  \caption{Hesse configuration}
  \label{fig:hesse}
\end{figure}

\begin{ex}\rm (1) $(4_3,6_2)$ is the complete quadrangle.

(2) $(9_4,12_3)$ is a Hesse configuration (see Figure 1).


(3) $(12_4,16_3)$ is a Reye configuration.

(4) For any $d\ge 3$, $C(d)=(3d_d,d^2_3)$ is a Ceva configuration.
\end{ex}


\section{Galois covers}\label{Galois}

Let $I\subset K[x,y,z]$ be an artinian ideal generated by forms of degree $d\ge 3$. We denote by $\mu(I)$
the minimal number of generators of $I$.
In this section,  we establish the main properties on  Galois covers
$$\varphi _I:\PP^2\longrightarrow \PP^{\mu(I)-1}$$ with cyclic Galois
group $\ZZ/d\ZZ$ represented by a matrix $$M_{a,b,c}:=\begin{pmatrix}
e^a&0&0\\
0&e^b&0\\
0&0&e^c
\end{pmatrix} $$  where $e$ is a primitive $d$-th root of $1$, and $0\le a\le b\le c\leq d-1$.
We observe that in order  that the subgroup of $GL(3)$ generated by $M_{a,b,c}$ is cyclic of order $d$,  the greatest common divisor of $a,b,c,d$ has to be $1$.
We note that the surface $S=\varphi_I(\PP^2)$ is the weighted projective plane $\PP(a,b,c)$, that has been extensively studied in the context of classification  of singularities of surfaces (\cite{R}, \cite{BPV}, \cite{CLS}).

First, for sake of completeness, we
give the proof that the artinian ideal $I$ associated to such Galois cover is
monomial. Indeed, we have:

\begin{thm} \label{galoismonomial}
 Fix $3\le d\in \ZZ$, $e$ a primitive $d$-th root of 1 and
$M_{a,b,c}$
a representation of $\ZZ/d\ZZ$. The ideal $I\subset K[x,y,z]$ generated
by all forms of degree $d$ invariant
under the action of $M_{a,b,c}$
is monomial. In particular, any GT-system is monomial.
\end{thm}
\begin{proof} Consider $m_1,\ldots ,m_s$ all monomials invariants under
the action of $M_{a,b,c}$
and assume there is $F\in I\setminus \langle m_1,\ldots ,m_s\rangle $. Write
$F$ as a sum of monomials of degree $d$:
$F=m_{i_1}+\cdots +m_{i_t}$. We proceed by induction on $t$. Set
$m_{i_1}=x^{\alpha}y^{\beta}z^{\gamma}$.
Since $M_{a,b,c}^rF=F$ for $0\le r\le d-1$ we have that $e^{ra
\alpha}x^{\alpha}e^{rb \beta}y^{\beta}e^{rc \gamma}z^{\gamma}$
is a summand of $F$ for $0\le r\le d-1$. Therefore,
$x^{\alpha}y^{\beta}z^{\gamma}e^{a \alpha+ b \beta+ c
\gamma}(e^0+e^1+\cdots +e^{d-1})
$ is a summand of $F$ but this is zero.
\end{proof}

\begin{prop}\label{galois_wlp} Fix an integer $d\ge 3$ and a
representation $M_{a,b,c}$
of $\ZZ/d\ZZ$. Let $I\subset K[x,y,z]$ be the ideal generated by all
monomials of degree $d$ invariant under the action of $M_{a,b,c}$. If $\mu (I)\le
d+1$ then $I$ is a GT-system, i.e. $ I$ fails WLP
 from degree $d-1$ to degree $d$.
\end{prop}

\begin{proof}  Since $\mu(I) \le d+1$ we can apply Theorem \ref{teathm}
and hence we only have to check that for a general linear form $L\in (K[x,y,z])_1$ the map $\times L
:(K[x,y,z])_{d-1}=(K[x,y,z]/I)_{d-1}\longrightarrow(K[x,y,z]/I)_d$ is not injective. Since $I$ is a monomial ideal it is enough to check the failure of the injectivity for $L=x+y+z$ (see \cite[Proposition 2.2]{MMN}).
This is equivalent to prove that  given  $L= x+y+z$ there exists a form
$C_{d-1}$ of degree $d-1$ such that $L\cdot C_{d-1}\in I$.
Consider $C_{d-1}=(e^ax+e^by+e^cz)(e^{2a}x+e^{2b}y+e^{2c}z)\cdots
(e^{(d-1)a}x+e^{(d-1)b}y+e^{(d-1)c}z).$
The polynomial $F=L\cdot C_{d-1}$ is invariant under the action of
$M_{a,b,c}$. Therefore, it belongs to $I$ and we have proved what we want.
\end{proof}

\begin{rem}\label{proportionality} \rm{We observe that, up to proportionality,
after replacing $a,b,c$ by $0, b-a, c-a$, we can assume $a=0$.  So from now on we will consider always actions defined by matrices of the form $M_{a,b}=\begin{pmatrix}
1&0&0\\
0&e^a&0\\
0&0&e^b
\end{pmatrix} $} with $e$ a primitive $d$-th root of 1 and $GDC(a,b,d)=1$.
\end{rem}

\begin{thm}\label{upper} Fix an integer $d\ge 3$ and let $I\subset
K[x,y,z]$ be the ideal generated by all monomials of degree $d$ invariant
under the action of $M_{a,b}:=\begin{pmatrix}
1 &       0 &      0\\
0 &     e^a &     0\\
0   &   0 &      e^b\end{pmatrix},$
with $1\le a\le b \le d-1$ and $GCD(a,b,d)=1$. Then, $I$ is a $GT$-system.
\end{thm}

\begin{proof} We know that $I$ is a monomial ideal (see Theorem
\ref{galoismonomial})  and any monomial of degree $d$ can be written in the form
$x^{d-m-n}y^mz^n$ with $m,n\ge 0$ and $m+n\le d$. Clearly
$x^{d-m-n}y^mz^n$ is invariant under the action of $M_{a,b}$ if and only if
$am+bn \equiv 0 \pmod{d}$.  It is easy to prove (see for instance \cite{Mc}, Ch. 3) that this linear congruence has $d$ (resp. $d+1$)
incongruent solutions if $a\neq b$  (resp. if $a=b$), which implies that $\mu(I) \le d+1.$  The thesis follows from Proposition  \ref{galois_wlp}.
\end{proof}


\section{Minimality of GT-systems}\label{minimality}

We fix an integer $d\ge 3$ and a representation $M_{a,b}:=\begin{pmatrix}
1 &       0 &      0\\
0 &     e^a &     0\\
0   &   0 &      e^b\end{pmatrix}$ of the cyclic group $\ZZ/d\ZZ$. We are going to study the minimality of  all monomial GT-systems $I$ generated by all monomials of degree $d$ invariant under the action of $M_{a,b}$.
By the proof of Proposition \ref{galois_wlp}, proving the minimality of the GT-system $I$ is equivalent to proving that the monomials of degree $d$ invariant under the action of $M_{a,b}$ all appear with non-zero coefficient in the development of the product of linear forms
$(x+y+z)(x+e^ay+e^b z)(x+e^{2a} y+e^{2b} z)\cdots (x+e^{(d-1)a} y+e^{(d-1)b}z)$. This is  not at  all a trivial problem, that we will address by expressing this last product as the determinant of a suitable circulant matrix, then we will exploit the basic properties of circulants (\cite{KS} and \cite{LWW}). So, let us start this section by recalling the definition and the properties on circulant matrices needed in the sequel. The reader should read \cite{W} and \cite{Malenf} for more details.

\begin{defn} \rm A $d\times d$ circulant matrix is a matrix of the form
$$ Circ(v_0,\ldots,v_{d-1}):=\begin{pmatrix} v_0 & v_1 & \cdots & v_{d-2} & v_{d-1} \\  v_{d-1} & v_0 & \cdots & v_{d-1} & v_{d-2} \\ \vdots & \vdots & \cdots & \vdots & \vdots \\
v_1 & v_ 2 & \cdots & v_{d-1} & v_0
\end{pmatrix}$$ where successive rows are circular permutations of the first row $v_0,\ldots,v_{d-1}$.
\end{defn}

A circulant matrix $ Circ(v_0,\ldots,v_{d-1})$ is a particular form of a Toeplitz
matrix i.e. a matrix whose elements are constant along the diagonals. A circulant matrix $ Circ(v_0,\ldots,v_{d-1})$ has $d$
eigenvalues, namely, $v_0 + e^pv_1 + e^{2p}v_2 + \cdots + e^{p(d-1)}v_{d-1}$, $0\le p \le d-1$, where $e$ is a primitive $d$-th root of unity. Therefore,
\begin{equation} \label{det_cir} \begin{array}{rcl} \det(Circ(v_0,\ldots,v_{d-1})) & = &  \left | \begin{matrix} v_0 & v_1 & \cdots & v_{d-2} & v_{d-1} \\  v_{d-1} & v_0 & \cdots & v_{d-1} & v_{d-2} \\ \vdots & \vdots & \cdots & \vdots & \vdots \\
v_1 & v_ 2 & \cdots & v_{d-1}& v_0  \end{matrix} \right | \\ \\
& = & \prod _{j=0}^{d-1} (v_0 + e^jv_1 + e^{2j}v_2 + \cdots + e^{j(d-1)}v_{d-1}). \end{array} \end{equation}

The product on the right hand side in the equation (\ref{det_cir}), when expanded out, contains ${2d-1}\choose d$ terms and it is still an open problem to find an efficient  formula for the coefficients and decide whether they are zero or not. Let us relate this problem to our problem of determining the minimality of GT-systems.
For any integer $d\geq 3$ and $1\leq a<b\leq d$, we consider the $d\times d$ circulant matrix \begin{equation*}A_d^{a,b}=Circ(x,0,\ldots , 0, y,0, \ldots 0, z,0,\ldots,0)\end{equation*}
where  $y$ is in the position of index $a$ and $z$ in the position of index $b$. We have $$\det (A_d^{a,b})= \prod _{j=0}^{d-1} (x + e^{aj}y + e^{bj}z).$$
 The determinant of $A_d^{a,b}$ is therefore exactly  the product we are interested in and
we want to prove that all monomials of degree $d$ invariant under the action of $M_{a,b}$ appear with non-zero coefficient in $\det(A_d^{a,b})$.  Let us now summarize what is known about the coefficients in the left hand side of the equation (\ref{det_cir}).  To this end we express the determinant of a $d\times d$ circulant matrix as follows:

$$ \det(Circ(v_0,\ldots,v_{d-1}))  = \sum _{0\le a_0\le \cdots \le a_{d-1}\le d-1}c_{a_0\cdots a_{d-1}}v_{a_0}\cdots v_{a_{d-1}}.$$
\begin{prop}\label{coefI} With the above notation, it holds:
\begin{itemize}
\item[(1)] If $a_0+a_1+\cdots +a_{d-1}\not\equiv 0 \pmod{d}$, then $c_{a_0\cdots a_{d-1}}=0 $.
\item[(2)] If $d$ is prime and $a_0+a_1+\cdots +a_{d-1}\equiv 0 \pmod{d}$, then $c_{a_0\cdots a_{d-1}}\ne 0 $.
\end{itemize}
\end{prop}

\begin{proof} (1) See \cite{Malenf}, Theorem 1 or \cite{W}, Proposition 10.4.3.

(2) See \cite{Malenf}, Corollary 4 or \cite{W}, Chapter 11.

\end{proof}

\begin{rem} \rm The converse of Proposition \ref{coefI}(2) is not true if $d$ is not prime. Indeed, for $d=6$ we have $c_{0,0,1,3,3,5}=0$, and for $d=10$ we have $c_{0,0,0,0,1,21,1,3,6,8}=0$ (see, for instance, \cite{W}, p. 123).
\end{rem}

In \cite{LWW}, Loehr, Warrington and Wilf  addressed  the problem of determining whether $c_{a_0\cdots a_{d-1}}\ne 0 $ when condition (1) in Proposition  \ref{coefI}  is satisfied and there are only three distinct non-zero elements in $(v_0,v_1,\ldots ,v_{d-1})$. They  obtained the following result.

\begin{prop}\label{coefII} Fix  integers $d\ge 3$ and $d-1\ge a\ge 2$ and consider the $d\times d$ circulant matrix $A_d^{a}=Circ(x, y,0, \ldots 0, z,0,\ldots ,0)$ where $z$ is located  at the position of index $a$. Then, we have $$\det (A_d^a)=\prod _{j=0}^{d-1} (x + e^{j}y + e^{aj}z)=\sum _{0\le m,n\le d \atop m+n\le d}c_{m,n}x^{d-m-n}y^m z^n$$ and $c_{m,n}\ne 0$ if and only if $m+an\equiv 0 \pmod{d}$.
\end{prop}

\begin{proof} See \cite{LWW}, Theorem 2.
\end{proof}

We are now ready to state our main result concerning the minimality of GT-systems. We have

\begin{thm}\label{minimal}
Let $d\ge 3$ be an integer and let $I\subset
K[x,y,z]$ be the GT-system generated by all monomials of degree $d$, which are  invariant
under the action of $M_{a}=\begin{pmatrix}
1 &       0 &      0\\
0 &     e &     0\\
0   &   0 &      e^a\end{pmatrix}.$ Then, $I$ is  a minimal Togliatti system.
\end{thm}
\begin{proof}
It immediately follows from Proposition \ref{coefII}. In particular, all generalized classical Togliatti systems are minimal.
\end{proof}

In Section \ref{d prime}, we will prove that, if $d=p^r$ is a power of a prime $p\ge 3$, then the action on $K[x,y,z]_d$ of a representation $M_{a,b,c}$ of $\ZZ/d\ZZ$ is equivalent to the action on $K[x,y,z]_d$ of a representation $M_{\alpha }$ of $\ZZ/d\ZZ$ for a suitable $2\le \alpha \le p^r-1$ and, hence, for $d=p^r$, a power of a prime $p\ge 3$, {\em all} GT-systems are minimal. As we will see in Section \ref{general}, for a general integer $d$ with a prime factorization $d=p_1^{\alpha _1}\cdots p_r^{\alpha _r}$ with $r\ge 2$,  it is no longer true that
 all actions of $\ZZ/d\ZZ$ on $K[x,y,z]_d$ can be  represented by a matrix of the form $M_{\alpha }$ for a suitable $2\le \alpha \le d-1$ (see Remark \ref{different_action}). For a general integer $d$, we will only classify the actions of $\ZZ/d\ZZ$ on $K[x,y,z]_d$   represented by a matrix of the form $M_{\alpha }$ for a suitable $2\le \alpha \le d-1$. By Theorem \ref{minimal}, the GT-systems associated to these actions are always minimal.

Based on the above theorem and our computation, we are led to pose the following conjecture:

\begin{conj}
Fix an integer $d\ge 3$ and let $I\subset
K[x,y,z]$ be the $GT$-system generated by all monomials of degree $d$ invariant
under the action of $M_{a,b}=\begin{pmatrix}
1&0&0\\
0&e^a&0\\
0&0&e^b
\end{pmatrix} $
with $1\le a\le b \le d-1$ and $GCD(a,b,d)=1$. Then, $I$ is a minimal Togliatti system.
\end{conj}


\section{Classification of GT-systems: The case $d$ prime or power of a prime.}\label{d prime}

The goal of this section is to classify  all GT-systems $I\subset
K[x,y,z]$ generated by forms of degree $d$, where $d$ is a prime integer
or a power of a prime; i.e. classify all monomial Togliatti systems
 $I\subset K[x,y,z]$ generated by forms of degree $d$ whose associated
regular map is a Galois covering with cyclic Galois group $\ZZ/d\ZZ$.

We have to study the action  on $K[x,y,z]_d$ by the matrices of the form
$M_{a,b,c}$
with $0\leq a\leq b\leq c\leq d-1$. Our aim is first to characterize the
monomials, and then the  homogeneous artinian ideals, generated by forms
of degree $d$,  which are invariant under these actions. Therefore we
give the following definition.

\begin{defn}\label{def:eq}  We will say that   the actions on
$K[x,y,z]_d$ of two representations $M_{a,b,c}$ and $M_{a',b',c'}$ of
$\ZZ/d\ZZ$  are equivalent if the homogeneous artinian ideals, generated
by forms of degree $d$  which are invariant under their actions,
coincide up to  a permutation of the variables.
\end{defn}

First of all we observe that since we are interested in artinian ideals
$I\subset K[x,y,z]_d$ failing WLP from degree $d-1$ to degree $d$ we can
always assume that the three exponents are different, i.e. $0\le a< b<
c\le d-1$ because we have:

\begin{lem} Let $I\subset K[x,y,z]_d$ be an ideal invariant under the
action of $M_{a,b,c}.$
If $a=b$ or $b=c$ then I has the WLP.
\end{lem}
\begin{proof} Assume $a=b$ (analogous argument works for $b=c$). Then,
$I$ contains $(x,y)^d$ and $\times (x+y+z): (K[x,y,z])_{d-1}\longrightarrow
(K[x,y,z]/I)_{d}$ is surjective. Therefore, $I$ has WLP.
\end{proof}

\vskip 4mm
\subsection{$d$ prime}

Assume first that $d\geq 3$ is an odd prime.  Let $1, e, e^2, \ldots,
e^{d-1}$ be the $d$-th roots of $1$. Then, for any $i=1, \ldots, d-1$,
$e^i$ is a primitive root of $1$.

As observed in Remark \ref{proportionality}, we can assume $a=0$. Moreover we can
denote by $e$ the primitive $d$-th root of $1$ at the second position in
the diagonal of the matrix, in other words, we can assume $b= 1$. So we
are reduced to study the action on $K[x,y,z]_d$ by the representations
$M_a =\begin{pmatrix}
1 &       0 &      0\\
0 &     e &     0\\
0   &   0 &      e^a\end{pmatrix}$ of $\ZZ/d\ZZ$
for $a=2,\ldots,d-1$.

\begin{rem}\label{def:b} \rm
Let $\alpha$ be an integer with $2\leq \alpha\leq d-1$. We will associate to $\alpha$ two positive integers 
$b$ and $k$ defined as follows.
Since $d$ and $\alpha$ are
coprime,  there is an expression $b'\alpha -k'd=1$, with $b',k'\in \ZZ$.
Any other pair of coefficients occuring in a similar expression is of the form $(b'+nd, k'+n\alpha)$,
with $n\in \ZZ$. 

Let $b:=\min\{b'\in \ZZ
\mid  b'\alpha -k'd=1, k'\in \ZZ, b'>0\}$. Then the class of $b\pmod{d}$ is equal to $\frac{1}{\alpha}$ in
$\ZZ/d\ZZ$.

Let $k$
be the unique integer such that $b\alpha -kd=1$. It results that
 $k>0$. Moreover $1<b<d$. Indeed $b\neq 1$ because $\alpha<d$, and
$b\neq d$ otherwise $(\alpha-k)d=1$.
\end{rem}

\begin{lem} \label{lem:b}
\begin{enumerate}
\item If $\alpha=2$, then $\frac{1}{\alpha}=\frac{d+1}{2}$, $k=1$;
\item if $\alpha=d-1$, then $\frac{1}{\alpha}=d-1$, $k=d-2$;
\item if $\alpha<d-1$ then $\frac{1}{\alpha}\neq \alpha$.
\end{enumerate}
\end{lem}
\begin{proof} The first two assertions are straightforward. For the
third one,  $\alpha=1/\alpha$ gives $\alpha^2=1$ and
 since $d$ is prime, $d$ must divide $\alpha-1$ or $\alpha+1$,
but both cases are impossible.
\end{proof}
\bigskip

We consider now two matrices $M_a$ and $M_b$ and we want to understand
when their actions on $K[x,y,z]_d$ are equivalent.

\begin{prop}\label{prop:equivmat}
Assume  $2\leq a, b\leq d-1$. 
Then
the actions on $K[x,y,z]_d$ of the representations
$M_a$, $M_b$ of $\ZZ/d\ZZ$ are equivalent  if and only if either
$b=d-a+1$ or $(e^a)^b=e$.
\end{prop}
\begin{proof} Let $x^{d-m-n}y^mz^n$ be a monomial of degree $d$. This
monomial is invariant for the action of the first matrix if and only if
$m+an\equiv 0 \pmod{d}$. But this congruence is equivalent to
$(d-m-n)+(d-a+1)n\equiv 0 \pmod{d}$, which means that the monomial
$x^my^{d-m-n}z^n$ obtained from  $x^{d-m-n}y^mz^n$ interchanging the
roles of $x$ and $y$ is invariant for the action of the second matrix.
The second equivalence corresponds to interchanging the roles of $y$ and
$z$, because if $(e^a)^b=e$ then $m+an\equiv 0\pmod{d}$ if and only if
$bm+n\equiv 0 \pmod{d}$.
\end{proof}

\begin{cor}\label{cor:equivmat}
The three matrices
$M_2$, $M_{\frac{d+1}{2}}$, $M_{d-1}$ define equivalent actions on the
homogenous polynomials of degree $d$.
\end{cor}
\begin{proof}  It immediately follows from Proposition
\ref{prop:equivmat} because $d-1=d-2+1$ and $e^{(\frac{d+1}{2})^2}=e$.
\end{proof}

\begin{rem} \rm
Assume  $2< a< d-1$, $a\neq \frac{d+1}{2}$. Then there exists one and
only one $b$, $2<b<d-1$, $b\neq \frac{d+1}{2}$ such that  $(e^a)^b=e$.
Indeed,
$b$ is nothing else than $1/a$ in $\ZZ/d\ZZ$ (see Remark \ref{def:b}
and  Lemma \ref{lem:b}).
\end{rem}

Before stating next theorem, we remark that any prime number $d\geq 5$
can be written in one of the forms $6n-1$ or $6n+1$. 

Let us premise a lemma, 
which will also be used in the next section.

\begin{lem} \label{lem:d} Let $p\geq 5$ be a prime, let $d=p^r$ with $r\geq 1$. Then there
exists an integer $a$, $2<a<d-1$, such that $a(d-a+1)\equiv 1 \pmod{d}$
if and only if $p=6n+1$.
\end{lem}
\begin{proof} The relation $a(d-a+1)\equiv 1 \pmod{d}$ holds true if and
only if $a^2-a+1\equiv 0 \pmod{d}$, i.e. $a$ is a root modulo $d$ of the
cyclotomic polynomial $\phi_6(x)=x^2-x+1$.  This means that the cyclic
group $\ZZ/6\ZZ$ of the roots of order $6$ of $1$ is contained in
$(\ZZ/p^r\ZZ)^*$. This happens if and only if $6$ divides the cardinality of
$(\ZZ/p^r\ZZ)^*$, which is $\varphi(p^r)=p^{r-1}(p-1)$, if and only if $p\equiv 1
\pmod 6$. The thesis follows.
\end{proof}

\begin{thm}\label{thm:number} Let $d\geq 5$ be a prime number. Then the
number of non-equivalent actions of representations $M_a$, $2\le a\le
d-1$, of  $\ZZ/d\ZZ$ on $K[x,y,z]_d$ is equal to $n$ if $d=6n-1$, and to
$n+1$ if $d=6n+1$.
\end{thm}
\begin{proof} By Corollary \ref{cor:equivmat}, there is a class with
three elements corresponding to $a=2, \frac{d+1}{2}, d-1$. If $d=6n+1$, there is a class with two elements
$a_0, d-a_0+1$, where $a_0$ is a root of $\phi_6(x)$ (Lemma
\ref{lem:d}). The remaining values of $a$ are subdivided in classes with
$6$ elements that are $a, b, c,d-a+1,  d-b+1, d-c+1$, where $b=1/a
\pmod{d}$ and $c=1/(d-a+1) \pmod{d}$. This immediately gives the number
of classes, in view of Proposition \ref{prop:equivmat}.
\end{proof}

\begin{rem}\label{crossratio} \rm
The six numbers $a, b, c,d-a+1,  d-b+1, d-c+1$, where $b= 1/a \pmod{d}$
and $c=1/(d-a+1) \pmod{d}$, are the values of the cross-ratio with the
same $j-$invariant, modulo $d$. The three values $2, \frac{d+1}{2}, d-1$
correspond to the harmonic cross-ratio, and the pair $a_0, d-a_0+1$ to
the equianharmonic cross-ratio.
\end{rem}

The matrices in Corollary \ref{cor:equivmat}  define the generalized
classical Togliatti systems (see Definition \ref{genclassicalTogliatti}).
We will see that the  other classes we have just found define all
GT-systems for $d$ prime. To this end we will study  the degree $d$
monomials invariant under these actions.

\begin{ex} \rm
Here are the classes of integers $\ge 2$ defining equivalent actions
$M_a$ on forms of degree $d$, for the primes $d=5,\ldots,17$.

$d=5$, one class: $(2,3,4);$

$d=7$, two classes: $(2,4,6)$, $(3,5);$

$d=11$, two classes: $(2,6,10)$, $(3,4, 5,7,8, 9)$;

$d=13$, three classes: $(2,7,12)$, $(4,10)$, $(3, 5,6,8, 9, 11)$;

$d=17$, three classes: $(2,9,16)$, $(3,6,8,10,12,15)$, $(4,5,7,11,13,14)$.
\end{ex}

Let us fix a prime $d$ and an integer $a$, $2\leq a\leq d-1$. We want to
count the number of monomials of degree $d$ which are invariant under
the action of the matrix $M_a$.

\begin{thm}\label{numberinv} Let $d\geq 3$ be a prime. Then for any
integer $a$, with $2\leq a\leq d-1$, the number of  invariant monomials
under the action of $M_a$  is $3+{\frac{d-1}{2}}$.
\end{thm}
\begin{proof} The monomial  $x^{d-m-n}y^mz^n$, with $m,n\geq 0$,
$m+n\leq d$, is invariant  under the action of $M_a$
if and only if $m+an\equiv 0\  \pmod{d}$. So we look for the pairs
$(m,n)$ of non-negative integers  such that $0\leq m+n\leq d$ and there
exists  a relation $m+an=kd$ for a suitable integer $k\geq 0$.  The
pairs $(0,0), (d,0), (0,d)$ are  clearly of this type; they correspond
to the monomials $x^d$, $y^d$, and $z^d$. So, we assume from now on
$m<d$, $n<d$ and  $m+n>0$.  We observe  that, if $m+n=d$ and  the
corresponding monomial $y^{d-n}z^n$ is invariant under the action of
$M_a$, then $d+(a-1)n=kd$,  i.e. $(a-1)n=(k-1)d$, which is impossible
under our assumption that   $n<d$. Moreover the monomials of the form
$x^{d-m}y^m$, $x^{d-n}z^n$ are certainly not invariant, the first ones
because $e^m\neq 1$, the second one because $an$ cannot be a multiple of
$d$. So from now on we can also assume $m,n>0$ and $m+n<d$.

We begin looking for pairs $(m,n)$ with $m=1$. We apply Lemma
\ref{lem:b} to $a,d$: there exist unique $k,n_1$ with $dk-an_1=1$,
$k>0$, $0<n_1<d$.
So we obtain a unique invariant monomial $x^{d-1-n_1}yz^{n_1}$, of
degree $d$  with $m=1$. We easily check that the inequality $1+n_1<d$ is
also true.

We take now another $m$ and define $n_m:\equiv n_1m  \pmod{d}$, i.e. the
remainder of the division of $n_1m$ by $d$. The pair $(m,n_m)$ defines a
monomial of degree $d$ invariant under the action of $M_a$ if and only if
$m+n_m\leq d$.  So we want to count for how many $m$'s with $1\leq m\leq
d-1$ the condition $m+n_m\leq d$ holds true.
We consider the table
\begin{align*}
&1      &n_1  &&n_1\\
&2      &2n_1&&n_2\\
&3      &3n_1&& n_3\\
&\dots&\dots&&\dots\\
&m     &mn_1&&n_m\\
&\dots &\dots&&\dots\\
&d-2   &(d-2)n_1&&n_{d-2}\\
&d-1   &(d-1)n_1&&n_{d-1}
\end{align*}
In the third column there are the remainders of the division by $d$ of
the elements of the second column. We observe that they are all
different: if $mn_1=hd+r$ and $m'n_1=h'd+r$, then  $(m-m')n_1=(h-h')d$,
which is impossible because $m-m'$ and $n_1$ are $<d$. Hence, the first
and the third column contain the same
numbers.

Now look at the last row: note that $n_{d-1}=d-n_1$ and that
$(d-1)+(d-n_1)=2d-(1+n_1)\leq d$ if and only if $d\leq 1+n_1$, a
contradiction. Therefore, the last line does not give an invariant
monomial. Now we consider the second and the last but one rows: the
corresponding pairs $(m,n)$ are respectively $(2,n_2)$ and $(d-2,
d-n_2)$, and $2d-(2+n_2)<d$ if and only if $2+n_2>d$, which means that of
these two pairs one and only one defines a   monomial of degree
$d$ invariant under the action of $M_a$. We can repeat this argument for
any $m$, concluding that one and only one of the pairs $(m, n_m)$,
$(d-m, d-n_m)$ gives a monomial of degree $d$ invariant under the action
of $M_a$. Since there are $d-1$ of these pairs,   we get $\frac{d-1}{2}$
invariant monomials of degree $d$ in addition to $x^d, y^d, z^d$. This
concludes the proof.
\end{proof}

\begin{rem}\label{monom} \rm
Note that any monomial of degree $d$ different from $x^d, y^d, z^d$ and
invariant under the action of $M_a$ must contain all variables with
strictly positive exponent. In particular, by the smoothness criterion
\cite[Proposition 3.4]{MM}, this implies that the $GT$-systems for $d$
prime are all smooth.
\end{rem}

\begin{ex} \rm
For $d=7$ or $d=11$, there is only one class, in addition to the
generalized classical Togliatti systems.

If $d=7$ and $a=3$, we get the invariant monomials $x^7, x^4yz^2,
x^2y^4z, xy^2z^4, y^7, z^7$.

If $d=11$, we have $5$ invariant monomials. For instance, for
$a=3$, we get:
$$x^{11}, x^3yz^7, x^6y^2z^3, xy^4z^6, x^4y^5z^2, x^2y^8z, y^{11}, z^{11}.$$
\end{ex}

From now on, we denote by $I_a\subset K[x,y,z]$ the artinian ideal
generated by all monomials of degree $d$ invariant under the action of
$M_a$.  Given a prime integer $d\ge 3$ and $2\le a\le d-1$, we denote by
$k$ and $n_1$ the unique integers $k,n_1>0$ with  $dk-an_1=1$ and
$n_1<d$. As in the proof of Theorem \ref{numberinv}, for any integer $m$, we define $n_m:=n_1m \  \pmod{d}$, i.e.
the remainder of the division of $n_1m$ by $d$. We have seen:

\begin{prop} For any prime integer $d\ge 5$, set $d=2k+1$. Take $2\le
a\le d-1$. Then, the following holds:

\begin{itemize}
\item[(1)] $\mu(I_a)=k+3$. Moreover, there exist integers $1\le
i_1<i_2<\cdots <i_k\le d-1$ such that
$$I_a=(x^d,y^d,z^d,f_{i_1},f_{i_2},\ldots ,f_{i_k})$$
where we use the notation $f_{i_j}=x^{d-i_j-m_{i_j}}y^{i_j}z^{m_{i_j}}$ for $1\le j\le k$.
\item[(2)] $I_2$, $I_{d-1}$ and $I_{k+1}$ coincide up to permutation of
the variables.
\item[(3)] If $a\ne 2, d-1, \frac{d+1}{2}$, and $\frac{1}{a}\not\equiv d-a+1
\pmod{d}$, then, $I_a$, $I_b$. $I_c$, $I_{d-a+1}$, $I_{d-b+1}$ and
$I_{d-c+1}$, where $b=\frac{1}{a}\pmod{d}$ and $c=\frac{1}{d-a+1}\pmod{d}$,
coincide up to permutation of the variables.
    \item[(4)] If $\frac{1}{a}\equiv d-a+1\pmod{d}$ then $I_a$ and
$I_{d-a+1}$ coincide up to permutation of the variables.
     \item[(5)] $I_a$ is a smooth monomial minimal GT-system.
\end{itemize}
\end{prop}
\begin{proof} (1) It follows from the proof of Theorem \ref{numberinv}.

(2) It follows from Corollary \ref{cor:equivmat}.

(3) See the proof of Theorem \ref{thm:number}.

(4) It is a direct application of Proposition \ref{prop:equivmat}.

(5)   It follows from Theorems \ref{upper} and \ref{minimal}, and Remark \ref{monom} for the smoothness.
\end{proof}

\vskip 4mm
\subsection{$d$ a power of a prime}

We consider now the cases when $d=p^r$ is a power of a prime $p$.  First
we observe that arguing as in the case when $d$ is an odd prime we can show that
the action on $K[x,y,z]_d$ of a representation $M_{a,b,c}$  of
$\ZZ/d\ZZ$ is always equivalent to the action on $K[x,y,z]_d$ of a
representation $M_{\alpha }$  of $\ZZ/d\ZZ$ for a suitable $\alpha $,
$2\le \alpha \le p^r-1$. Therefore, it is enough to classify the actions
on $K[x,y,z]_d$ of  representations $M_{\alpha }$  of $\ZZ/d\ZZ$.

\begin{thm}\label{power} Let $d=p^r$, where $p$ is a prime, and $r\geq
2$ an integer.
Then the number of classes of non-equivalent actions of representations
$M_a$ of $\ZZ/d\ZZ$ on $K[x,y,z]_d$ is  as follows:
\begin{enumerate}
\item[{(1)}] if $p=2$,  $r=2$, there is only one class; if $p=2$, $r>2$
there  are $(d/4)+1$ classes, of which three classes have two elements and the
remaining classes have $4$ elements;
\item[{(2)}] if $p=3$ or $p=6n-1$, then there are one class with three
elements, $(d-p)/2p$ classes with $4$ elements and $(dp-2d-3p)/6p$
classes with six elements;
\item[{(3)}] if $p=6n+1$, then there are one class with three elements,
one class with two elements, $(d-p)/2p$ classes with $4$ elements and
$(dp-2d-5p)/6p$ classes with six elements.
\end{enumerate}
\end{thm}
\begin{proof}
(1) If $p=2$, there are exactly three values of $a$, $2\leq a\leq d-1$,
such that $a^2\equiv 1 \pmod{d}$: $a=d-1$, $(d-2)/2$ and $(d+2)/2$. Each
of them gives rise to a class formed by   two elements $a, d-a+1$.
For any other value of  $a$, in the pair $(a, d-a+1)$ there is an even
and an odd number, which admits inverse $b$. So, the class is formed by
$a, b, d-a+1, d-b+1$.

Assume now that $p$ is odd. Then $a^2\equiv 1 \pmod{d}$ if and only if
$a=d-1$, which produces the class with three elements: $2$, $d-1$ and
$(d+1)/2$. From Lemma \ref{lem:d} it follows that the equation $a(d-a+1)\equiv 1 \pmod{d}$ has a
solution $a_0$ only if $p=6n+1$.
So we have a class with two elements $a_0, d-a_0+1$ only if $d=6n+1$.
Then for any multiple of $p$, $px$, we have a class with $4$ elements of the
form $px, d-px+1, y, d-y+1$, where $y$ is the inverse of $d-px+1$ modulo
$d$. It is easy to check that in this class there is also exactly one
other multiple of $p$ i.e. $d-y+1$. Finally, the remaining elements
distribute themselves in classes of $6$ elements of the form $a,b,d-a+1,
x, d-b+1, d-x+1$, where $b, x$ are the inverses $\pmod{d}$ of $a$ and
$d-a+1$, respectively. The thesis follows by a straightforward
computation, taking into account that the number of non-zero multiples
of $p$ in $\ZZ/d\ZZ$ is $(d/p)-1$.
\end{proof}

\section{Classification of GT-systems: The general case}\label{general}

We consider now the  case of a general integer $d$. We will restrict our attention
to actions represented
by matrices of the form $M_a$.

For seek of completeness we start recalling a couple of well known
results on arithmetic equations.

\begin{lem}\label{lem:first} Let $d=2^\alpha
p_1^{\alpha_1}p_2^{\alpha_2}\ldots p_r^{\alpha_r}$ be the prime
factorization of $d$, where   $p_i$ are distinct odd primes,
$r\geq 0$, $\alpha\geq 0$,  and $\alpha_i> 0$ for $1\le i\le r.$
The number of solutions of the equation $x^2\equiv 1 \pmod d$ is:
\begin{enumerate}
\item $2^r$ if $0\le \alpha \le 1$;
\item $2^{r+1}$ if $\alpha=2$;
\item $2^{r+2}$ if $\alpha\ge 3$.
\end{enumerate}
\end{lem}
\begin{proof} See \cite{Enc}, a060594.
\end{proof}

\begin{lem}\label{lem:second} Let $d=p_0^{\alpha_0}
p_1^{\alpha_1}\ldots p_r^{\alpha_r}$ be the prime
factorization of $d$, where $r\geq 0$, $p_0=2$, $p_1=3$,  $p_i>3$ for $i\ge 2$,
 $\alpha_0, \alpha_1\geq 0$,  and $\alpha_i> 0$ for $2\le i\le r.$
The equation $x^2-x+1 \equiv 0 \pmod d$ is compatible if and
only if $\alpha_0=0$, $0\le \alpha_1 \le 1$ and $p_i\equiv 1 \pmod 6$ for any
$i=2,\ldots,r$, and in this case there are exactly $2^{r-1}$ solutions.
\end{lem}
\begin{proof} First of all we recall from Lemma \ref{lem:d} that for any prime integer $p\ge
5$, $x^2-x+1 \equiv 0 \pmod{p^k}$ is compatible  if and only if $p\equiv 1
\pmod 6$. Moreover, in this case we have two solutions.

For $p=2$, there is no solution of the equation $x^2-x+1 \equiv 0 \pmod{2^a}$ and for $p=3$  there is a solution of the equation $x^2-x+1
\equiv 0 \pmod{3^b}$ if and only if $b=1$ and in this case there is
only one solution.

Putting altogether, we obtain the thesis.
\end{proof}

We want to count now the number of non-equivalent representations $M_a$
of $\ZZ/d\ZZ$ on $K[x,y,z]_d$.
\begin{lem} \label{descriptionclasses} An equivalence class of
representations $M_a$ of $\ZZ/d\ZZ$ on $K[x,y,z]_d$ can have $2, 3, 4$
or $6$ elements. Precisely:
\begin{itemize}
\item an equivalence  class with two elements is formed by $a, d-a+1$,
where one of the following happens:

(i) $a^2\equiv 1$, and $d-a+1$ is non invertible in $\ZZ/d\ZZ$;

(ii) $a(d-a+1)\equiv 1 \pmod d$;

(iii) $a, d-a+1$ are both non invertible in $\ZZ/d\ZZ$;
\item an equivalence class with three elements is formed by $2, d-1,
\frac{d+1}{2}$; such class  exists if and only if $d$ is odd;
\item an equivalence class with four elements is formed by $a, d-a+1, y,
d-y+1$, where $y$ is the inverse of $d-a+1 \pmod d$, and $a, d-y+1$ are
non invertible in $\ZZ/d\ZZ$;
\item an equivalence class with six elements is formed by $a,b,c, d-a+1,
d-b+1, d-c+1$ where $b=1/a$, $c=1/(d-a+1)$; all six elements are
invertible in $\ZZ/d\ZZ$.
\end{itemize}
\end{lem}
\begin{proof}  We recall, from Proposition \ref{prop:equivmat}, that
the  class containing an element $a$   always contains also $d-a+1$,
and, if $a$ is invertible, also $1/a$.

 So to have a class with only two elements, one of the following three
possibilities must occur: $1/a\equiv a$, or $1/a\equiv d-a+1$, or $a$
and $d-a+1$ are both non invertible. In the first case, it follows that the
two integers $d$ and $d-a+1$ are not coprime. Indeed if $1/a\equiv a$
in $\ZZ/d\ZZ$, there exists $k\in \ZZ$ such that $a^2-1=(a-1)(a+1)=kd$.
So each prime divisor of $d$ divides either $a-1$ or $a+1$, and at least
one divides $a-1$. Therefore $d-(a-1)$ is not coprime with $d$, this
gives case (i). In the second case and third case, it is clear that the
class does not contain any other element.  For the class with three
elements see Corollary \ref{cor:equivmat}. The shape of the classes with
four or six elements are analogous to those discussed in Theorem
\ref{power}.
\end{proof}

Let us denote by $N_i$ the number of classes with $i$ elements, and
$N_{21},  N_{22}, N_{23}$ respectively the numbers of classes with two
elements of type (i),(ii),(iii). Our next goal is to compute them.
\begin{prop}\label{N2}
Let $d=p_0^{\alpha_0}
p_1^{\alpha_1}\ldots p_r^{\alpha_r}$ be the prime
factorization of $d$, where $r\geq 0$, $p_0=2$, $p_1=3$,  $p_i>3$ for $i\ge 2$,
 $\alpha_0, \alpha_1\geq 0$,  and $\alpha_i> 0$ for $2\le i\le r.$ It holds:
\begin{itemize}
\item[(1)] $N_{21}=\begin{cases}  2^r-2 & \text{ if }  \alpha_0 =0;\\ 2^r-1 & \text{ if }  \alpha_0 = 1;\\
2^{r+1}-1 & \text{ if } \alpha_0=2;\\
2^{r+2}-1 & \text{ if }\alpha_0\ge 3.
\end{cases}$
\vskip 2mm
\item[(2)]  $N_{22}=\begin{cases}  2^r & \text{ if } \alpha_0=0, 0\le
\alpha_1 \le 1 \text{ and } p_i\equiv 1 \pmod 6 \text{ for any } i=2,\ldots ,r;
\\
0 & \text{ otherwise.}
\end{cases}$
\vskip 2mm
\item[(3)] If $\alpha_0=0$, then $N_{23}=\sum_{k=2}^r (-1)^{k}(2^{k-1}-1)\sum_{0\le
i_1<\cdots<i_{k}\le r}\frac{d}{p_{i_1}\cdots p_{i_k}}$. If $\alpha_0>0$, then
the expression is analogous but the sum goes from $2$ to $r+1$.
\end{itemize}
In addition, $N_2=N_{21}+N_{22}+N_{23}$.
\end{prop}
\begin{proof} (1) is a consequence of Lemma \ref{lem:first}, but we have to subtract $1$ because we do not want to consider the solution $1$. (2) is a consequence of Lemma \ref{lem:second}. We prove now (3). Let $p_i,p_j$ be two prime factors of $d$. We claim that
the number of pairs
$(a,d-a+1)$, where $a$ is divisible by $p_i$ and $d-a+1$ by $p_j$, is
$d/(p_ip_j)$. Indeed, the pair $(a=p_ix, d-a+1=p_jy)$ is as we want if
we can  write $d+1$ as $p_ix+p_jy$, with $x, y$ positive integers.
We can write $p_i\lambda-p_j\mu=1$, and we choose $\lambda$ minimum
$>0$, hence $\mu>0$. So the pairs $x,y$ that work are obtained in this way:

$x=(1+d)\lambda-mp_j>0$,

$y=-(1+d)\mu+mp_i>0$,

\noindent with $m>0$. Thus, we must count the number of positive $m$ such that the
two inequalities are satisfied.
We must have
$(1+d)\mu/p_i < m < (1+d)\lambda/p_j$,
which is equivalent to
$$p_j(d/p_ip_j)\mu+(\mu/p_i) < m < p_i(d/p_ip_j)\lambda + (\lambda/p_j).$$
The number of $m$'s satisfying this inequality is
$p_i(d/p_ip_j)\lambda-p_j(d/p_ip_j)\mu=d/p_ip_j$, which proves the claim.

So  the number of these pairs, as $i,j$ vary, is $\sum_{0\leq i<j\leq
r}d/p_ip_j$.

But if $r \geq 2$, some of the pairs have been counted twice. So
we have to subtract the contribution of these pairs. For any choice of
three prime factors $p_i, p_j, p_h$, the contribution is
$3d/(p_ip_jp_h)$, because it is equal to the number of pairs $(a,
d-a+1)$ such that $a$ is divisible by one among $p_i, p_j, p_h$, and
$d-a+1$ by the product of the other two. We continue in this way, taking
the contribution of $2,3,\ldots$ prime factors of $d$ with alternate
signs. The coefficient in the contribution of $k$ primes is the number
of partitions of $k$ as sum of two positive numbers, i.e.
$\sum_{i=1}^{\lfloor k/2 \rfloor}{k\choose i}=(2^k-2)/2$.
\end{proof}

\begin{rem}\label{N3} \rm For any $d=p_0^{\alpha_0} p_1^{\alpha_1
}p_2^{\alpha_2}\cdots p_r^{\alpha_r}$   with $p_0=2$, $p_1=3$, $p_i>3$ for $i\ge2$,
$r\geq 0$, $\alpha_0, \alpha_1 \ge 0$ and $\alpha_i> 0$ for $i\ge 2$, we have
$N_3= \begin{cases}  1 & \text{ if } \alpha_0=0;
\\
0 & \text{ otherwise.}
\end{cases}$
\end{rem}

\begin{prop} \label{N4} Let $d$ be an integer.
Let
$\varphi(d)$ be the Euler's function that counts the positive integers
up to $d$ that are relatively prime to $d$. Then the following  holds:
\begin{itemize}
\item[(1)] \begin{equation*}N_4=\frac{ d-1-\varphi(d)-N_{21}-2N_{23}}{2};\end{equation*}
\item[(2)] \begin{equation*} N_6=\frac{d-2-2N_2-3N_3-4N_4}{6}.\end{equation*}
\end{itemize}
\end{prop}
\begin{proof} (1) By Lemma \ref{descriptionclasses} the equivalence
classes with four elements are formed by $a, d-a+1, y, d-y+1$, where $y$
is the inverse of $d-a+1 \pmod d$, and $a, d-y+1$ are non invertible in
$\ZZ/d\ZZ$. So, we have to compute the number of non-invertible elements
not involved in  the equivalence classes with 2 or 3 elements. Since the
Euler function  $\varphi(d)$ counts the invertible elements in
$\ZZ/d\ZZ$. i.e. the cardinality of $(\ZZ/d\ZZ)^*$, we get the required formula for $N_4.$

(2) It immediately follows from the fact that
$2N_2+3N_3+4N_4+6N_6=d-2$.\end{proof}

Putting altogether we have

\begin{thm} \label{number} Let $d=p_0^{\alpha_0}
p_1^{\alpha_1}\cdots p_r^{\alpha_r}$ be the prime
factorization of $d$, where $r\geq 0$, $p_0=2$, $p_1=3$,  $p_i>3$ for $i\ge 2$,
 $\alpha_0, \alpha_1\geq 0$,  and $\alpha_i> 0$ for $2\le i\le r.$
The number of classes of non-equivalent actions of representations $M_a$
of $\ZZ/d\ZZ$ on $K[x,y,z]_d$ is  $N_2+N_3+N_4+N_6$.
\end{thm}
\begin{proof} It immediately follows from Lemma \ref{descriptionclasses}.
\end{proof}

Let us illustrate  our result with a couple of concrete examples.

\begin{ex}\rm  (i) If $d=825=3\cdot 5^2\cdot 11$, we get the following numbers:
$N_3=1,
N_2=N_{21}+N_{22}+N_{23}=6+0+80=86,
N_4=129,
N_6= 22.$

(ii) If $d=42=2\cdot 3\cdot 7$, then $N_3=0, N_2=N_{21}+N_{22}+N_{23}=3+0+9=12, N_4=4, N_6=0$.

(iii)
If $d=210=2\cdot 3\cdot 5\cdot 7$, then $N_3=0, N_2=N_{21}+N_{22}+N_{23}=7+57=64, N_4=20, N_6=0$.
\end{ex}

In Theorem \ref{upper} we have obtained un upper bound, valid for any
$d$, on the number of  monomials of degree $d$ invariant under the
action of $M_a$. Next Proposition gives a lower bound on the same
number.

\begin{prop}\label{lower} Let $d=p_1^{\alpha_1}\cdots p_r^{\alpha_r}$,
where $p_1<p_2<\cdots<p_r$ are primes, $\alpha_i>0$ for all $i$ and
$r\geq 2$.
Let $I_a$ be the ideal generated by all polynomials of degree $d$
invariant under the  action of $M_a$, $2\leq a\leq d-1$.  Then, $I_a$ is
a monomial minimal GT-system and  $\mu(I_a)\geq \lfloor d/2\rfloor+3$.
\end{prop}

\begin{proof} By Theorems \ref{galoismonomial}, \ref{upper} and \ref{minimal}, $I_a$ is a monomial minimal GT-system. Thus, it only remains
to compute $\mu (I_a)$. We distinguish two cases.

If we consider an equivalence class that can be represented by a number
$a$ coprime with $d$, we can argue as in the proof of Theorem
\ref{numberinv}. For any $m$ with $1\leq m\leq d-1$, we define the
number $n_m>0$, the minimal positive integer such that $an_m+m$ is a
multiple of $d$ of the form $kd, k>0$. Note that for $m=1$ we have $n_1
\neq d-1$. Otherwise it would be $a(d-1)+1=kd$ with $1\leq k \leq a-1$,
so $d=(a-1)/(a-k)$, which is impossible. We note that the integer
$n_{d-m}$ associated to $d-m$ is equal to $d-n_m$. Therefore $m+n_m\leq
d$ if and only if $(d-m)+n_{d-m}\geq d$.

We have to count the integers $m$ such that $m+n_m\leq d$. It is enough
to analyze only  $\lfloor d/2\rfloor$ pairs.
If, for any $m, n_m \neq d-m$, then only one of the possibilities: $m+n_m<d$
or $(d-m)+(d-n_m)<d$ is true,  and we can conclude that
$\mu(I_a)=\lfloor d/2\rfloor +3$. Otherwise, if $n_m=d-m$ for some $m$, then
both pairs $(m,n_m=d-m)$ and $(d-m, m)$ give invariant monomial and
$\mu(I_a)>\lfloor d/2\rfloor+3.$

Assume now that $a$ and $d-a+1$ are both non invertible in $\ZZ/d\ZZ$,
i.e. that both are not coprime with $d$. Let $g=\gcd(d,a)>1$ and write
$a=g\alpha$ and   $d=g\delta$.   
We look for the pairs $(m,n)$ with $m\geq 0$, $n\geq 0$, $m+n\leq d$,
such that $m+an=kd$, for some $k$. This last relation is possible only
if $m=0$ or $m$ is a multiple of  $\gcd(d,a)=g$. So $m$ has to be of the
form $yg$, with $0\leq y\leq \delta$.

If $m=0$, then $0=kd-an$, i.e. $n=kd/a=k\delta/\alpha$, which implies
that $\alpha$ divides $k$, so $n=x\delta$, with $0\leq x\leq g$.

If $m=g$, we look for $n,k$ such that $g=kd-an$. Let $\bar{k},\bar{n}$
be the  integers with $\bar{k}>0$ minimum such that
$g=\bar{k}d-\bar{n}a$, i.e. $1=\bar{k}\delta-\bar{n}\alpha$. So also
$\bar{n}>0$. The other integers $k,n$ are of the form $\bar{k}+x\alpha,
\bar{n}+x\delta$, with $x\in \ZZ$. From them we get the pairs $(m,n)=(g,
\bar{n}+x\delta)$. Note that $0\leq n\leq d$ if and only if $0\leq x\leq
g-1$. Indeed  $\bar{n}+x\delta\leq d$ if and only if $x\leq
(d-\bar{n})/\delta=g-(\bar{n}/\delta)$; moreover from the minimality of
$\bar{k}$, we deduce $\bar{n}-\delta<0$, which gives $0<\bar{n}/\delta<1$.

Let now $m=yg$, with $1\leq y\leq \delta$. We clearly have the pairs
$(m,n)=(yg, y\bar{n}+x\delta)$. The condition $0\leq
n=y\bar{n}+x\delta\leq d$ becomes $-y\bar{n}/\delta\leq x\leq
g-y\bar{n}\delta$.

Altogether we have $d+2$ pairs: $g+1$ pairs with $m=0$, $g$ pairs with
$m=yg$ for any $1\leq y\leq \delta-1$, and $(d,0)$. We have to impose
now the condition $m+n\leq d$, i.e. $yg+y\bar{n}+x\delta\leq d$ to the
pairs different from $(d,0), (0,d), (0,0)$. But it is easy to check that
this last inequality is equivalent to
$(\delta-y)g+(\delta-y)\bar{n}+(g-\bar{n}-x)\delta\geq d$. It is clear
that $0\leq \delta-y\leq d$. Moreover $-(\delta-y)\bar{n}/\delta\leq
g-\bar{n}-x\leq g-(\delta-y)\bar{n}/\delta$ is equivalent to
$(d-g)-y\bar{n}\leq \delta x\leq d-y\bar{n}$, so  the pair $(\delta-y)g,
(\delta-y)\bar{n}+(g-\bar{n}-x)\delta$ is in the admissible range. We
conclude that there are at least $\lfloor d/2\rfloor+3$ invariant monomials.

\end{proof}

\begin{rem}\rm If $\mu(I_a)>\lfloor d/2\rfloor+3$, with $a$ invertible
modulo $d$, then $ mn_1+m = m(n_1+1)$ must be multiple of $d$. This
observation allows to construct a whole class of examples with $d=2^k$,
where the first $m$ such that $n_m=d-m$ is $2$, and $n_1+1=d/2$. It is
enough to take $a=(d/2)+1$ and then $n_1=(d/2)-1$. We can compute that
$\mu(I_a)=3d/4+2$.   In this case, all pairs $(m, d-m+1)$ with $m$ even
satisfy $n_m=d-m$. The first example is for $d=8, a=5, n_1=3, \mu(I_5)=8$.
For $d=16$, we get $a=9, n_1=7$, and $\mu(I_9)=14$.
\end{rem}

\begin{rem}\label{different_action} \rm
Not all actions  of a cyclic group on $K[x,y,z]_d$ can be represented by a matrix of the form $M_a$, if
$d$ has $3$ or more prime factors.
For instance if $d=42=2\cdot 3\cdot 7$ and $e$ is a primitive root of order $42$ of $1$, then the action of $M_{0,3,7}$ is not equivalent to the action of a matrix
of the form $M_a$ for any $a$, since it has the following set of monomials invariant under the action of $M_{0,3,7}$ and involving only two variables. They are: $$x^{36}z^6, x^{30}z^{12}, x^{24}z^{18}, x^{18}z^{24}, x^{12}z^{30}, x^6z^{36},x^{28}y^{14}, x^{14}y^{28},y^{21}z^{21}.$$  If we now consider the monomials invariant under the action of $M_{a}$ and involving only two variables,
we have monomials involving $x$ and $z$, and monomials involving $y$ and $z$ but never monomials involving $x$ and $y$.

\end{rem}


\section{Geometric properties of generalized classical Togliatti systems}\label{geomproperties}

This  section is entirely devoted to study the geometric properties
of the rational surfaces associated to  generalized classical Togliatti
systems.
To this end, we fix an integer $d$, $3\le d\in \ZZ$. We write $d=2k+\epsilon
$, $0\le \epsilon \le 1$ and we consider the monomial artinian ideal

$$I_d=(x^d,y^d,z^d,x^ky^kz^{\epsilon},\ldots
,x^2y^2z^{d-4},xyz^{d-2})\subset K[x,y,z].$$
Set $\PP^2=\Proj(K[x,y,z])$ and $\PP^{k+2}=\Proj(K[x_0,x_1,\ldots ,
x_{k+2}])$. We have seen in Theorem \ref{minimal} that $I_d$ is a minimal Togliatti
system which defines a Galois cover
$$\varphi _{I_d}:\PP^2\longrightarrow \PP^{k+2}$$ with Galois cyclic
group $\ZZ/d\ZZ$ represented by $\begin{pmatrix}
1&0&0\\
0&e^2&0\\
0&0&e
\end{pmatrix} $  where $e$ is a primitive $d$-th root of 1.

We denote by $S_d\subset \PP^{k+2}$ the rational surface image of $\varphi _{I_d}$.  We will prove that its homogeneous ideal $I(S_d)$
is generated by $1+{k\choose 2}$ quadrics if $d$ is even  and by
${k\choose 2}$ quadrics and $k$ cubics if $d$ is odd. To achieve our
result we need some basic facts on determinantal ideals that we recall
now for seek of completeness.

\begin{defn} \rm  Let  $\cA$ be a homogeneous $m\times n$ matrix, $m\ge n$. We denote by  $I(\cA)$ the
ideal of $R=K[x_0,x_1,\ldots ,x_N]$ generated by the  maximal minors of $\cA$ and we say that
$I(\cA)$
 is  a
\emph{determinantal} ideal  if $\depth I(\cA)=m-n+1$.
\end{defn}

Any determinantal ideal $I\subset R$ is Cohen-Macaulay (i.e. $\pd
(R/I)=\codim (I)$ or, equivalent $\dim (R/I)=\depth (R/I)$) with a
minimal free resolution given by the so called {\em Eagon-Northcott
complex}. In fact, we denote by
$f
:F\longrightarrow G$ the morphism of free graded $R$-modules of
rank $t+c-1$ and $t$, defined by the homogeneous $t \times (t+c-1)$
matrix $\cA$ associated to the determinantal ideal $I\subset R$ of
codimension $c$. The Eagon-Northcott complex
$$
0 \longrightarrow \bigwedge^{t+c-1}F \otimes S_{c-1}(G)^*\otimes
\bigwedge^tG^*\longrightarrow \bigwedge^{t+c-2} F \otimes S
_{c-2}(G)^*\otimes
\bigwedge ^tG^*\longrightarrow \cdots $$
$$\longrightarrow
\bigwedge^{t}F \otimes
S_{0}(G)^*\otimes \bigwedge^tG^* \longrightarrow R \longrightarrow
R/I\longrightarrow 0
$$
gives a minimal free $R$-resolution of $R/I$.

\begin{thm}\label{generators} Let $I_d\subset K[x,y,z]$ be a generalized
classical Togliatti system and set $R=K[x_0,x_1,\ldots ,x_{k+2}]$. Then,
the following holds:
\begin{itemize}
\item[(1)] If $d=2k+1$ is odd then $I(S_d)=I_2(\cA )$ where $$\cA
=\begin{pmatrix}x_3 & x_ 4 & \cdots & x_{k+1} & x_{k+2} & x_0x_1\\ x_4 &
x_ 5 & \cdots & x_{k+2} & x_{2} & x_3^2 \end{pmatrix} .$$
In particular, $I(S_d)$ is generated by ${k\choose 2}$ quadrics and $k$
cubics.
\item[(2)] If $d=2k$ is even then $I(S_d)=I_2(\cB )+(x_0x_1-x_3^2)$ where
$$\cB=\begin{pmatrix} x_3 & x_ 4 & \cdots & x_{k+1} & x_{k+2} \\ x_4 &
x_ 5 & \cdots & x_{k+2} & x_{2}  \end{pmatrix} .$$
In particular, $I(S_d)$ is generated by $1+{k\choose 2}$ quadrics.
\end{itemize}
\end{thm}
\begin{proof} (1) The morphism $\varphi _{I_d}:\PP^2\longrightarrow
\PP^{k+2}$ is defined by sending $$(x,y,z)\mapsto
(x^d,y^d,z^d,x^ky^kz,\ldots ,xyz^{d-2}).$$ Therefore, we easily check
that $I_2(\cA)\subset I(S_d)$. The equality follows from the fact that
$\dim (R/I(S_d))=\dim(R/I_2(\cA))=3$ and $\degree (I_2(\cA))=\degree (I(S_d))=d$.

(2) It is analogous.
\end{proof}

\begin{thm} With the above notation, the following holds:
\begin{itemize}
\item[(1)] If $d=2k+1$ is odd then $I(S_d)$ is a determinantal ideal
(and, hence, Cohen-Macaulay) with the following minimal free resolution:
$$0\longrightarrow R(-k-2)^k\longrightarrow \cdots  \longrightarrow
R(-4)^{3{k\choose 4}}\oplus R(-5)^{3{k\choose 3}} \longrightarrow $$
 $$R(-3)^{2{k\choose 3}}\oplus R(-4)^{2{k\choose 2}} \longrightarrow
R(-2)^{{k\choose 2}}\oplus R(-3)^{k}\longrightarrow R \longrightarrow
R/I(S_d) \longrightarrow 0.$$
\item[(2)] If $d=2k$ is even then $I(S_d)$ is a  Cohen-Macaulay ideal with
the following minimal free resolution:
$$
0\longrightarrow  R(-k-2)^k \longrightarrow R(-k-1)^{(k-1){k\choose
k-1}} \oplus  R(-k)^k
\longrightarrow \cdots \longrightarrow $$
 $$R(-5)^{2{k\choose 3}}\oplus R(-4)^{3{k\choose 4}} \longrightarrow
R(-4)^{k\choose 2}\oplus R(-3)^{2{k\choose 3}} \longrightarrow
R(-2)^{1+{k\choose 2}}\longrightarrow R \longrightarrow
R/I(S_d)\longrightarrow  0.
$$
\end{itemize}
\end{thm}
\begin{proof} (1) By Theorem \ref{generators}(1) we know that
$I(S_d)=I_2(\cA )$ where $$\cA =\begin{pmatrix}x_3 & x_ 4 & \cdots &
x_{k+1} & x_{k+2} & x_0x_1\\ x_4 & x_ 5 & \cdots & x_{k+2} & x_{2} &
x_3^2 \end{pmatrix} .$$ Associated to $\cA$ we have a morphism
$F:=R(-1)^k\oplus R(-2)\stackrel{f}{\longrightarrow} G:=R^2$ of free
graded $R$-modules
 and by the Eagon-Northcott complex $I(S_d)$ has the following minimal
free  $R$-resolution:
 $$0\longrightarrow \bigwedge ^{k+1}F\otimes S_{k-1}G^*\longrightarrow
\cdots \longrightarrow \bigwedge ^{4}F\otimes S_{2}G^*\longrightarrow
\bigwedge ^{3}F\otimes G^*\longrightarrow \bigwedge ^{2}F\longrightarrow
R \longrightarrow  R/I(S_d)\longrightarrow  0$$
which gives what we want.

(2)   By Theorem \ref{generators}(2) we know that  $I(S_d)=I_2(\cB
)+(x_0x_1-x_3^2)$ where $$\cB=\begin{pmatrix} x_3 & x_ 4 & \cdots &
x_{k+1} & x_{k+2} \\ x_4 & x_ 5 & \cdots & x_{k+2} & x_{2}
\end{pmatrix} .$$ By the Eagon-Northcott complex $I_2(\cB )$ has the
following minimal free $R$-resolution
$$0 \longrightarrow R(-k)^k
\longrightarrow R(-k+1)^{(k-1){k\choose k-1}}\longrightarrow\cdots
\longrightarrow
R(-4)^{3{k\choose 4}} \longrightarrow $$
$$
R(-3)^{2{k\choose 3}} \longrightarrow
R(-2)^{k\choose 2} \longrightarrow R \longrightarrow
R/I_2(\cB)\longrightarrow  0.$$
To find a minimal free $R$-resolution of $I(S_d)=I_2(\cB )+(x_0x_1-x_3^2)$
we consider the short exact sequence
$$
0\longrightarrow I_2(\cB)(-2) \stackrel{\psi }{\longrightarrow}
I_2(\cB)\oplus R(-2)\stackrel{\rho}{\longrightarrow}  I(S_d)=I_2(\cB
)+(x_0x_1-x_3^2)\longrightarrow 0
$$
where $\psi(F)=(F(x_0x_1-x_3^2),F)$ and $\rho (G,H)=G-(x_0x_1-x_3^2)H$.
Applying the mapping cone procedure to the following diagram
$$
 \begin{array}{ccccccccccc}
 & & 0 & & 0 \\
  & & \downarrow & & \downarrow & &  \\
  & & R(-k-2)^k & & R(-k)^k\\
   & & \downarrow & & \downarrow & &  \\
 & &  \vdots & & \vdots \\
      & & \downarrow & & \downarrow & &  \\
  &&  R(-6)^{3{k\choose 4}} & &  R(-4)^{3{k\choose 4}}  \\
      & & \downarrow & & \downarrow & &  \\
& & R(-5)^{2{k\choose 3}} & &  R(-3)^{2{k\choose 3}} \\
 & & \downarrow & & \downarrow & &  \\
 & & R(-4)^{k\choose 2} & & R(-2)^{k\choose 2}\oplus R(-2) \\
 & & \downarrow & & \downarrow & &  \\
 0 & \longrightarrow & I_2(\cB)(-2) & \stackrel{\psi }{\longrightarrow}
&  I_2(\cB)\oplus R(-2) & \stackrel{\rho}{\longrightarrow} &  I(S_d) &
\longrightarrow & 0 \\
 & & \downarrow & & \downarrow &&   \\
 & & 0 & & 0
 \end{array}
$$
we get the minimal free $R$ resolution of $I(S_d)$:
$$
0\longrightarrow  R(-k-2)^k \longrightarrow R(-k-1)^{(k-1){k\choose
k-1}} \oplus  R(-k)^k
\longrightarrow \cdots \longrightarrow $$
 $$R(-5)^{2{k\choose 3}}\oplus R(-4)^{3{k\choose 4}} \longrightarrow
R(-4)^{k\choose 2}\oplus R(-3)^{2{k\choose 3}} \longrightarrow
R(-2)^{1+{k\choose 2}}\longrightarrow R \longrightarrow
R/I(S_d)\longrightarrow  0.
$$
\end{proof}

\subsection{Singularities of $S_d$}

We describe now the singular points of the surface $S_d$. The morphism
$\varphi_{I_d}$ is unramified outside the three fundamental points of
$\PP^2$: $E_0=[1,0,0]$, $E_1=[0,1,0]$, $E_2=[0,0,1]$. They are sent by
$\varphi_{I_d}$ to the singular points of $S_d$, $P_i:=\varphi(E_i)$,
$i=1,2,3$, that are cyclic quotient singularities: $P_0$, $P_1$ are  of type ${\frac{1}{ d}}(1,2)$ and
$P_2$ is  of type ${\frac{1}{ d}}(d-2,d-1)$.

\begin{prop} The points $P_0$, $P_1$ have multiplicity $k$ on $S_d$ and
tangent cone defined by the equations $\rk \cB<2$ and $x_1=0$ (resp.
$x_0=0$). The point $P_2$ is a double point with tangent cone the pair of planes of equations
$x_0x_1=x_3=\dots =x_k=x_{k+1}=0$.
\end{prop}
\begin{proof}
Let us compute first the tangent cone to $S_d$ at $P_0$. We work on the
affine chart $x_0=1$, with affine coordinates $x_1, \ldots, x_{k+2}$. If
$d$ is odd,  looking at the initial forms of the polynomials in the
ideal of $S_d$, it follows from Theorem \ref{generators}  that the
equations of the tangent cone are $\rk \cB<2$ and
$x_1x_4=x_1x_5=\dots=x_1x_{k+2}=x_1x_2=0$, which reduce to $\rk \cB<2$
and $x_1=0$.  If $d$ is even, we get immediately  the same equations.
The case of $P_1$ is similar.

We work then on the affine chart $x_2=1$ with affine coordinates $x_0,
x_1, x_3, \ldots, x_{k+2}$. If $d$ is odd, the tangent cone at $P_2$ is
defined by $\rk \cB'<2$, $x_3=x_4=\dots=x_{k+1}=x_0x_1=0$, where
$$\cB'=\begin{pmatrix} x_3 & x_ 4 & \cdots & x_{k+1}  \\ x_4 & x_ 5 &
\cdots & x_{k+2}  \end{pmatrix} .$$ This gives the equations
$x_3=x_4=\dots=x_{k+1}=x_0x_1=0$. If $d$ is even, we get directly the
same equations.
\end{proof}

\begin{rem} \rm
It is interesting to remark that the surface $S_d$ is also a Galois covering of $\PP^2$, with Galois
group $\ZZ/d\ZZ$. The covering map $S_d\to \PP^2$ composed with $\varphi_{I_d}$
is  $\PP^2\to\PP^2$, $[x,y,z]\to[x^d, y^d, z^d]$.
\end{rem}

\section{Relations with linear configurations}\label{config}

We establish now a link between GT-systems and linear configurations.

\begin{thm} \label{configurations} Fix an integer $d\ge 2$ and an
integer $a$, $2\le a\le d-1$. Let $I_a\subset K[x,y,z]$ be the
artinian ideal generated by all forms of degree $d$ invariant under the
action of $M_a$. Then, the following holds.
\begin{itemize}  \item[(1)] $I_a$ is a monomial minimal GT-system.

\item[(2)] For any linear form $L=\alpha x+\beta y+\gamma z$ we define
$L[j]:=\alpha x+\beta e^j y+\gamma e^{aj}z$, $j\ge 0$. It holds:
$$L\cdot \prod _{j=1}^{d-1}L[j]\in I_a.$$
\item[(3)] For any integer $0\le i\le d-1$ we define $L_{i,j}:=x+e^{i}y+e^{j}z$.
The set of $d^2$ lines $\{L_{i,j}\}_{0\le i,j\le d-1}$ gives rise to a
Ceva configuration $C(d)=(3d_d,d^2_3)$ of $d^2$ lines  and $3d$ points
with $d$ lines through each point and 3 points on each line.
\end{itemize}
\end{thm}

\begin{proof} (1)  follows from Theorems \ref{galoismonomial}, \ref{upper} and \ref{minimal}.

(2) Since $I_a$ fails WLP from degree $d-1$ to degree $d$, we know that
for any line $L=\alpha x+\beta y+\gamma z$ there exists a form $C_{d-1}$
of degree $d-1$ such that $F=L\cdot C_{d-1}\in I_a$. By definition $I_a$
is generated by the monomials of degree $d$ invariant under the action of
  $M_a$. Therefore, $F=L\cdot C_{d-1}$ is invariant under the action of
  $M_a$, which implies that $L[j]:=\alpha x+\beta e^j y+\gamma e^{aj}z$,
$j\ge 1$, divides $F$ and we conclude that $C_{d-1}$ factorizes as $
\prod _{j=1}^{d-1}L[j]$ which proves what we want.

(3) We have a set $\cL=\{L_{i,j}\}_{0\le i,j\le d-1}$ of $d^2$ lines in
$\PP^2$ and we consider $\cP=\{p_1,\cdots ,p_{3d}\}$ the following set
of points
$$\begin{matrix} p_1=(1,0,-1) & p_2=(1,0,-e) & p_3=(1,0,-e^2) & \cdots &
p_d=(1,0,-e^{d-1}) \\
 p_{d+1}=(0,1,-1) & p_{d+2}=(0,1,-e) & p_{d+3}=(0,1,-e^2) & \cdots &
p_{2d}=(0,1,-e^{d-1}) \\
 p_{2d+1}=(1,-1,0) & p_{2d+2}=(1,-e,0) & p_{2d+3}=(1,-e^2,0) & \cdots &
p_{3d}=(1,-e^{d-1},0)
\end{matrix}$$
The pair  $(\cP,\cL)$ is a Ceva configuration  $C(d)$ since we have
 $d^2$ lines and $3d$ points such that any line $L_{i,j}$ contains 3
points and through each point $p_s$ we have $d$ lines.
\end{proof}

 Associated to the above Ceva configuration $C(d)$ we have an
arrangement $\cH _{d}$ of $d^2+3$ lines $\{x,y,z\}\cup \{L_{i,j}\}_{0\le
i,j\le d-1}$ in $\PP^2$ and the derivation bundle $\cD _0$ defined as
the kernel of the jacobian map:

 $$0\longrightarrow \cD _0\longrightarrow\cO ^3_{\PP^2}\stackrel{\nabla
F}{\longrightarrow} \cO _{\PP^2}(d^2+2)$$

\vskip 2mm
\noindent where $F=xyz\prod_{0\le i,j\le d-1}L_{i,j}$  and
$\nabla F =(\frac{\partial F}{\partial x},\frac{\partial F}{\partial
y},\frac{\partial F}{\partial z})$. The arrangement
$\cH _{d}$ of
$d^2+3$ lines in $\PP^2$ is said to be {\em free with exponents }
$(a,b)$ if the derivation  bundle $\cD _0$ splits as $\cD_0=\cO
_{\PP^2}(-a)\oplus \cO _{\PP^2}(-b)$. So, we are let to pose the natural
question whether the arrangement $\cH _{d}$ of lines associated to the
Ceva configuration $C(d)$ is free. It holds:

\begin{prop} With the above notation the arrangement $\cH _d$ is a free
arrangement of $d^2+3$ lines in $\PP^2$ if and only if $3\le d\le 4$.
Moreover, $\cH_3$ is an arrangement of 12 lines free with exponent (4,7)
and  $\cH_4$ is an arrangement of 19 lines free with exponents (9,9).
\end{prop}

\begin{proof} If $\cH _d$  is free with exponents $(a,b)$ we necessarily
have $c_1(\cD _0)=a+b=d^2+2$ and $c_2(\cD _0)=ab$.  On the other hand,
if $b_h$ is the number of points of multiplicity $h$ in $\cD _0$, then we
have the following relation (see \cite{FV}, Remark 2.2):
$$ b_3+3b_4+6b_5+10b_6+\cdots ={d^2+2\choose 2}-c_2(\cD _0).$$

Therefore, if $\cH_d$ is a free arrangement necessarily $3\le d\le 4$. The fact that $\cH _3$ and $\cH _4$ are free arrangements of lines follows from \cite[Theorem 2]{FV}.
In addition,  $\cH_3$ is a free arrangement of 12 lines with 9 quadruple
base points. It is a free arrangement  with exponents (4,7). $\cH_4$ is
a free arrangement of 19 lines with 12 quintuple base points; it is a
free arrangement  with exponents (9,9).
\end{proof}

\begin{rem} \rm
The set of $3d$ points in a Ceva  configuration  $C(d)$ is the dual set
of the so called {\em Fermat arrangement} $\cF _d$:
$(x^d-y^d)(x^d-z^d)(y^d-z^d)=0$
which is an arrangement of $3d$ lines in $\PP^2$ with  $d^2$ triple
points and 3 points of multiplicity $d$. $\cF _d$ is a free arrangement
of lines in $\PP^2$ with exponents $(2d-2,d+1)$. In fact, let $\cD _0$
be derivation bundle  defined as
the kernel of the jacobian map:

 $$0\longrightarrow \cD _0\longrightarrow\cO ^3_{\PP^2}\stackrel{\nabla
F}{\longrightarrow} \cO _{\PP^2}(3d-1)$$

\vskip 2mm
\noindent where $F=(x^d-y^d)(x^d-z^d)(y^d-z^d)$.  We are looking for integers $a\le b$ such that  $c_1(\cD _0)=3d-1=a+b$ and $c_2(\cD _0)=ab={3d-1\choose 2}- d^2+3{d-1\choose 2}$.  Therefore, $(a,b)=(d+1,2d-2)$ and the freeness of $\cF _d$ follows from \cite[Theorem 2]{FV}.
\end{rem}


\end{document}